\documentclass[12pt,leqno]{article}
\usepackage{latexsym}
\usepackage{amsthm}
\usepackage{amsfonts}
\usepackage{amssymb}
\usepackage{amsmath} 
\usepackage{graphicx}
\usepackage{url} 
\usepackage{color}
\usepackage{indentfirst}

\usepackage[top=30truemm,bottom=30truemm,left=30truemm,right=30truemm]{geometry}

\theoremstyle{plain}
  \newtheorem{theorem}{Theorem}[section] 
  
  \newtheorem{proposition}{Proposition}[section]
  \newtheorem{lemma}{Lemma}[section]

\theoremstyle{remark}
  \newtheorem{remark}{Remark}[section]

\theoremstyle{definition}

\makeatletter

\@addtoreset{equation}{section}
\makeatother

\begin{document}

\markboth{Hideo Takaoka}{}

\title{On the growth of Sobolev norm for the cubic NLS on two dimensional product space}
\author{Hideo Takaoka\thanks{This work was supported by JSPS KAKENHI Grant Number 18H01129.}\\
Department of Mathematics, Kobe University\\
Kobe, 657-8501, Japan\\
takaoka@math.kobe-u.ac.jp}

\date{\empty}

\maketitle

\begin{abstract}
We obtain polynomial bounds on the growth in time of Sobolev norm of solutions to the cubic defocusing nonlinear Schr\"odinger equation on two dimensional product space.
We also give the angular improved bilinear Strichartz estimates for frequency localized functions, which estimates are used for enhancement of a smoothing estimates.
Such upper bounds for the growth of Sobolev norms measure the transfer of energy from low to high modes as time grows on. 
\end{abstract}

{\it $2020$ Mathematics Subject Classification Numbers.}
35Q55, 42B37.


\section{Introduction}\label{sec:introduction}
In this paper, we consider the cubic defocusing nonlinear Schr\"odinger equation on two dimensional product space
\begin{equation}\label{eq:NLS}
\begin{cases}
i\partial_t+\Delta u=|u|^2u,\\
u(0,z)=u_0(z)\in H^s(\mathbb{R}\times\mathbb{T}),
\end{cases}
\end{equation}
where $z=(x,y)\in \mathbb{R}\times \mathbb{T}$ and $\mathbb{T}=\mathbb{R}/2\pi\mathbb{Z}$.
The unknown function $u:\mathbb{R}\times\mathbb{R}\times\mathbb{T}\mapsto \mathbb{C}$ is a solution to \eqref{eq:NLS}, and obeys the integral equation
$$
u(t)= e^{it\Delta}u_0-i\int_0^te^{i(t-t')\Delta}\left[|u|^2u\right](t')\,dt'.
$$
We will focus on the defocussing nonlinearity as a simple model. 

As is known, the equation \eqref{eq:NLS} possesses at least two conservation laws; the energy conserved quantity 
\begin{equation}\label{eq:energy}
E[u](t)=\int_{\mathbb{R}\times\mathbb{T}}\left(\frac12|\nabla u(t,z)|^2+\frac14|u(t,z)|^4\right)\,dz
\end{equation}
and the mass conserved quantity
\begin{equation}\label{eq:mass}
M[u](t)=\int_{\mathbb{R}\times\mathbb{T}}|u(t,z)|^2\,dz.
\end{equation}
The conserved energy in the defocusing case is positive definite and controls the $H^1$ energy norm of solutions.  
Making use of conservation of energy \eqref{eq:energy} and mass \eqref{eq:mass} proves a priori bounds for solutions of \eqref{eq:NLS} as
\begin{equation*}\label{eq:H^1-apriori}
\sup_{t\in\mathbb{R}}\|u(t)\|_{H^1(\mathbb{R}\times\mathbb{T})}<\infty.
\end{equation*}
The aim of this paper is to develop the problem of growth in time of higher order Sobolev norms of solutions to \eqref{eq:NLS}.

There exist considerable amount of results in well-posedness literature for Euclidean space $\mathbb{R}^2$.
On the Euclidean space, one has dispersive estimates that allow us to obtain an $L^{\infty}$-estimate for the linearized solution.
One standard consequence of the dispersive estimates is the Strichartz estimates.
Such Strichartz estimates are fundamental and key tools for the analysis of the well-posedness problem.

In the periodic setting, such a dispersive estimate does not hold on one hand.
This makes the analysis of the well-posedness problem more difficult than the Euclidean setting.
In \cite{b1}, Bourgain investigated the Strichartz estimates for Schr\"odinger equations on two dimensional tori $\mathbb{T}^2$.
He proved $H^s~(s>0)$ well-posedness results for the cubic defocusing nonlinear Schr\"odinger equation on two dimensional torus $\mathbb{T}^2$ by using the $\varepsilon$-loss Strichartz estimates
\begin{equation}\label{eq:BourgainStrichartz}
\left\|e^{it\Delta}\phi_N\right\|_{L^{4}(\mathbb{T}^3)}\le C(\varepsilon) N^{\varepsilon}\|\phi_N\|_{L^2(\mathbb{T}^2)},
\end{equation}
where $\varepsilon>0$ and $\phi_N\in L^2(\mathbb{T}^2)$ with spatial frequencies $N$.
There is no way to compensate the loss of regularity in \eqref{eq:BourgainStrichartz}, in contrast with the Euclidean setting $\mathbb{R}^2$.
The upper bound on the growth in time of high Sobolev norms of global solutions on $\mathbb{T}^2$ was achieved by Bourgain \cite{b2}.
The growth of high Sobolev norms of solutions gives us quantitative estimate on the energy cascade phenomena, which captures the transfer of energy from low to high frequency modes.
He proved polynomial in time upper bounds on the Sobolev norm growth to solutions of the cubic defocusing nonlinear Schr\"odinger equation on $\mathbb{T}^2$
\begin{equation}\label{eq:bourgain-e}
\|u(t)\|_{H^s(\mathbb{T}^2)}\lesssim \langle t\rangle^{2(s-1)+}
\end{equation}
for $s>1$.
Moreover, he conjectured that if solutions exist global-in-time, the growth should be sub-polynomial in time, namely
\begin{equation*}\label{eq:bourgain-conjecture}
\|u(t)\|_{H^s(\mathbb{T}^2)}\le C(\varepsilon)\langle t\rangle^{\varepsilon}
\end{equation*}
for any $\varepsilon>0$.
The lower bounds for the growth of Sobolev norms in the cubic defocusing nonlinear Schr\"odinger equation was investigated in \cite{ckstt2}.

 \begin{remark}\label{rem:ckstt}
We give a review of the results on the lower bound of solution obtained in \cite{ckstt2}.
Let $s>1,~K\gg 1$ and $0<\delta\ll 1$ be given.
Then there exists a global smooth solution $u(t)$ to the cubic defocusing nonlinear Schr\"odinger equation on $\mathbb{T}^2$ and a time $T>0$ with
$$
\|u(0)\|_{H^s(\mathbb{T}^2)} \le \delta \quad \mathrm{and}\quad\|u(T)\|_{H^s(\mathbb{T}^2)} \ge K.
$$
This result is peculiar of the periodic setting.
In the whole space, Dodson \cite{d} proved that every solution in $L^2(\mathbb{R}^2)$ scatters and is finite in $L^2(\mathbb{R}^2)$.
\end{remark}

What happens to the problem in the product space setting $\mathbb{R}\times\mathbb{T}$ ?
In \cite{tatz}, we showed the local-in-time $L^4$-Strichartz estimates with fractional loss of derivatives
\begin{equation}\label{eq:L4Strichartz}
\left\|e^{it\Delta}\phi\right\|_{L^4([0,1]\times \mathbb{R}\times\mathbb{T})}\lesssim \|\phi\|_{L^2(\mathbb{R}\times\mathbb{T})}
\end{equation}
for $\phi\in L^2(\mathbb{R}\times\mathbb{T})$.
By \eqref{eq:L4Strichartz}, the global well-posednesss results for small data in $L^2(\mathbb{R}\times\mathbb{T})$ was established, see again \cite{tatz}.
The global well-posedness and scattering results was considered by Hani, Paisader, Tzvetokov and Visciglia \cite{hptv}.
They studied the resonance system associated to the cubic nonlinear Schr\"odinger equations on the general domain $\mathbb{R}\times\mathbb{T}^d$ for $1\le d\le 4$, in which the results on the asymptotic dynamics and the modified scattering operators for small data in weighed Sobolev spaces with certain norms were explored.
We refer to the paper \cite{cgz} by Chen, Guo and Zhao and the reference therein, where the global well-posedness and scattering results for the quintic nonlinear Schr\"odinger equation on the space $\mathbb{R}\times\mathbb{T}$. 

In light of \eqref{eq:bourgain-e}, Deng and Yang \cite{dy} obtained the bound on the growth in time of high Sobolev norms of global solutions to \eqref{eq:NLS}
\begin{equation*}\label{eq:dy}
\|u(t)\|_{H^s}\lesssim \langle t\rangle^{2(s-1)+}
\end{equation*}
for $u_0\in H^s(\mathbb{R}\times\mathbb{T}),~s>1$.
In this paper, we improve the polynomial upper bounds on the growth of solutions.
Our main results are stated as follows. 

\begin{theorem}\label{thm:growth}
Let $s>1$.
There exists a constant $C=C(\|u_0\|_{H^s})$ such that
$$
\|u(t)\|_{H^s}\le C\langle t\rangle^{\frac{s-1}{2}+}
$$
for any $t\in\mathbb{R}$ and any global-in-time solution $u(t)$ to \eqref{eq:NLS}. 
\end{theorem}

\begin{remark}
In contrast with the problem in $\mathbb{T}^2$, it is interesting to understand whether a growing-up solution to \eqref{eq:NLS} exists for the cubic nonlinear Schr\"odinger equation on the space $\mathbb{R}\times\mathbb{T}$.
It also remains a challenging open problem to establish the circumstances in which the situation like as Remark \ref{rem:ckstt} occurs.
\end{remark}

The proof of Theorem \ref{thm:growth} is established by the upside-down I-method, normal form reduction of the Hamiltonian, and bilinear Strichartz estimates.
The bilinear Strichartz estimates constructed here is angularly refined one, which is inspired by the previous work developed in \cite{ckstt1}.

Organization of the paper is as follows.
Section \ref{sec:notation} gives the notation used in the paper.
In Section \ref{sec:bilinear}, we continue our study of bilinear estimates associated with the linear Schr\"odinger equation.
A segmental angular analysis is used.
In Section \ref{sec:reductionenergy}, we will introduce the modified energy assigned with the cubic nonlinear Schr\"odinger equation in \eqref{eq:NLS}.
Section \ref{sec:proofTheorem} gives the proof of Theorem \ref{thm:growth}.
The idea behind the energy increment argument is that we observe the increment in the execution.

\section{Notation}\label{sec:notation}

In the paper, $\mathbb{T}=\mathbb{R}/2\pi\mathbb{Z}$.
We will use $c$ and $C$ to denote various time independent constants.
In the case that the implied constant depends on $a$, we write $c(a)$ or $C(a)$.
We use $a\lesssim b$ to denote an inequality of the form $a\le cb$.
Similarly, we write $a\sim b$ to mean $a\lesssim b$ and $a\gtrsim b$.
The notation $a+$ denotes $a+\varepsilon$ for an arbitrarily small $\varepsilon>0$.
Similarly, $a-$ denotes $a-\varepsilon$.   

We will use the shorthand notation $\zeta=(\xi,\eta)\in\mathbb{R}\times\mathbb{Z}$ for variables in physical space.
Define $(d\zeta)$ to be measuring-counting product measure on $\mathbb{R}\times\mathbb{Z}$
$$
\int a(\zeta)\,(d\zeta)=\sum_{\eta\in\mathbb{Z}}\int_{\mathbb{R}}a(\xi,\eta)\,d\xi,
$$
where $a(\zeta)=a(\xi,\eta)$ for $\zeta=(\xi,\eta)$.

We consider the functions on $f:\mathbb{R}\times\mathbb{T}\to \mathbb{C}$. 
Define the Fourier transform of a function $f$ defined on $\mathbb{R}\times\mathbb{T}$ by
$$
\widehat{f}(\zeta)=\frac{1}{2\pi}\int_{\mathbb{R}\times \mathbb{T}}e^{-iz\cdot\zeta}f(z)\,dz
$$
for $\zeta\in\mathbb{R}\times\mathbb{Z}$.
This leads to the Fourier inverse formula
$$
f(z)=\frac{1}{2\pi}\int e^{iz\cdot\zeta}\widehat{f}(\zeta)\,(d\zeta).
$$
By Plancherel, 
$$
\|f\|_{L^2(\mathbb{R}\times\mathbb{T})}^2=\int |\widehat{f}(\zeta)|^2\,(d\zeta).
$$
The usual properties of Fourier transform of functions on $\mathbb{R}\times\mathbb{T}$ hold as follows:
$$
\int\widehat{f}(\zeta)\overline{\widehat{g}(\zeta)}\,(d\zeta)=\int_{\mathbb{R}\times\mathbb{T}} f(z)\overline{g(z)}\,dz,
$$
$$
\widehat{fg}(\zeta)=\frac{1}{2\pi}\widehat{f}*\widehat{g}(\zeta),
$$
$$
\widehat{f*g}(\zeta)=2\pi \widehat{f}(\zeta)\widehat{g}(\zeta).
$$

We denote the symbol $\widehat{f}$ the space-time Fourier transform of a function $f$ defined on $\mathbb{R}\times\mathbb{R}\times \mathbb{T}$
$$
\widehat{f}(\tau,\zeta)=\frac{1}{(2\pi)^{3/2}}\int_{\mathbb{R}\times \mathbb{R}\times \mathbb{T}}e^{-it\tau-iz\cdot\zeta}f(t,z)\,dtdz
$$
for $(\tau,\zeta)\in\mathbb{R}\times \mathbb{R}\times\mathbb{Z}$, when this notation can not cause confusion.

We make use of two parameter spaces $X^{s,b}$ on $\mathbb{R}\times\mathbb{R}\times\mathbb{T}$ with the norm
\begin{equation*}
\|f\|_{X^{s,b}}=\left( \int \int_{\mathbb{R}}  \langle\zeta\rangle^{2s}\langle \tau-|\zeta|^2\rangle^{2b}|\widehat{f}(\tau,\zeta)|^2\,d\tau (d\zeta)\right)^{1/2}.
\end{equation*}
For the time interval $I$, we define the restricted spaces $X^{s,b}(I)$ by the norm
$$
\|f\|_{X^{s,b}(I)}=\inf\{\|F\|_{X^{s,b}(I)}\mid F|_{I\times \mathbb{R}\times\mathbb{T}}=f\}.
$$

\begin{remark}\label{rem:L4X}
It is well known that the Strichartz estimates related to \eqref{eq:L4Strichartz} is adapted into the spaces $X^{s,b}$ .
If $f$ is a function on $\mathbb{R}\times\mathbb{R}\times\mathbb{T}$, then
\begin{equation}\label{eq:L4X}
\|f\|_{L^4(I\times \mathbb{R}\times\mathbb{T})}\le C(|I|)\|f\|_{X^{0,1/2+}(I)}.
\end{equation}
The proof follows from a nice application of \eqref{eq:L4Strichartz}.
Stacking up dyadic level sets with respect to $\langle \tau- |\zeta|^2\rangle$ on which \eqref{eq:L4Strichartz} holds, we obtain \eqref{eq:L4X}. 
\end{remark}

\section{Angularly refined bilinear Strichartz estimates}\label{sec:bilinear}

In this section, we prove the bilinear estimates associated to the linear Schr\"odinger equation, induced with product space $\mathbb{R}\times\mathbb{T}$.

\begin{lemma}\label{lem:bilinear-o}
Let $1<N_1\le N_2$ and $M>1$.
Suppose that $\phi_{N_1},~\phi_{N_2}\in L^2(\mathbb{R}\times\mathbb{T})$ are the functions with spatial frequencies $N_1,~N_2$ respectively.
Then the space-time function
\begin{equation*}
F(t,z)=\int e^{-it(|\zeta_1|^2+|\zeta_2|^2)-iz\cdot(\zeta_1+\zeta_2)}1_{|\xi_1-\xi_2|\gtrsim M}\widehat{\phi}_{N_1}(\zeta_1)\widehat{\phi}_{N_2}(\zeta_2)\,(d\zeta_1)(d\zeta_2).
\end{equation*}
has the bound
$$
\|F\|_{L^2(\mathbb{R}\times\mathbb{R}\times\mathbb{T})}\lesssim \frac{N_1^{1/2}}{M^{1/2}}\|\phi_{N_1}\|_{L^2(\mathbb{R}\times\mathbb{T})}\|\phi_{N_2}\|_{L^2(\mathbb{R}\times\mathbb{T})}.
$$
\end{lemma}

\begin{proof}
By Placherel's identity, we have
\begin{equation*}
\begin{split}
& \|F\|_{L^2(\mathbb{R}\times\mathbb{R}\times\mathbb{T})}\\
= & \left\|\int \delta(\tau-|\zeta_1|^2-|\zeta_2|^2)\delta(\zeta-\zeta_1-\zeta_2)1_{|\xi_1-\xi_2|\gtrsim M} \widehat{\phi}_{N_1}(\zeta_1)\widehat{\phi}_{N_2}(\zeta_2)\,(d\zeta_1)(d\zeta_2)\right\|_{L^2(\mathbb{R}\times\mathbb{R}\times\mathbb{Z})}.
\end{split}
\end{equation*}
Further by using the Cauchy-Schwartz inequality, we get
\begin{equation*}
\begin{split}
& \|F\|_{L^2(\mathbb{R}\times\mathbb{R}\times\mathbb{T})}^2\lesssim \|\phi_{N_1}\|_{L^2(\mathbb{R}\times\mathbb{T})}^2\|\phi_{N_2}\|_{L^2(\mathbb{R}\times\mathbb{T})}^2\\
& \sup_{(\tau,\xi,\eta)\in\mathbb{R}\times\mathbb{R}\times\mathbb{Z}}\sum_{\eta_1\in\mathbb{Z}}1_{|\eta_1|\lesssim N_1}\int_{\mathbb{R}} 1_{|\xi_1-\xi/2|\gtrsim  M}\delta(\tau-\xi^2/2-2(\xi_1-\xi/2)^2-\eta_1^2-(\eta-\eta_1)^2)\,d\xi_1.
\end{split}
\end{equation*}
Therefore the proof of our claim reduces to show
$$
\int_{\mathbb{R}} 1_{|\xi_1-\xi/2|\gtrsim M}\delta(\tau-\xi^2/2-2(\xi_1-\xi/2)^2-\eta_1^2-(\eta-\eta_1)^2)\,d\xi_1\lesssim \frac{1}{M},
$$
which follows from the condition $|\xi_1-\xi/2|\gtrsim M$ in the integration process.
Hence we have the desired estimate.
\end{proof}

In keeping with the argument in Remark \ref{rem:L4X}, we obtain the following proposition from Lemma \ref{lem:bilinear-o}. 

\begin{proposition}\label{prop:1}
Let $1<N_1\le N_2$ and $M>1$.
Suppose that $f_{N_1},~f_{N_2}\in X^{0,1/2+}$ are the functions with spatial frequencies $N_1,~N_2$ respectively.
Then the space-time function
$$
F(t,z)=\int_{\mathbb{R}^2}\!\int e^{-it(\tau_1+\tau_2)-iz\cdot(\zeta_1+\zeta_2)}1_{|\xi_1-\xi_2|\gtrsim M}\widehat{f}_{N_1}(\tau_1,\zeta_1)\widehat{f}_{N_2}(\tau_2,\zeta_2)\,(d\zeta_1)(d\zeta_2)d\tau_1d\tau_2.
$$
has the bound
$$
\left\|F \right\|_{L^2(\mathbb{R}\times\mathbb{R}\times\mathbb{T})}\lesssim \frac{N_1^{1/2}}{M^{1/2}}\|f_{N_1}\|_{X^{0,1/2+}}\|f_{N_2}\|_{X^{0,1/2+}}.
$$
\end{proposition}

\begin{remark}\label{rem:pro} 
While the estimate in Proposition \ref{prop:1} is of interesting it own right, it is not difficult to prove
$$
 \left\|F \right\|_{L^2(\mathbb{R}\times\mathbb{R}\times\mathbb{T})}\lesssim \frac{N_1^{1/2}}{M^{1/2}}\|f_{N_1}\|_{X^{0,1/2+}}\|f_{N_2}\|_{X^{0,1/2+}}
$$
for the space-time function
$$
F(t,z)=\int_{\mathbb{R}^2}\!\int e^{-it(\tau_1+\tau_2)-iz\cdot(\zeta_1+\zeta_2)}1_{|\xi_1+\xi_2|\gtrsim M}\widehat{f}_{N_1}(\tau_1,\zeta_1)\widehat{\overline{f}}_{N_2}(\tau_2,\zeta_2)\,(d\zeta_1)(d\zeta_2)d\tau_1d\tau_2.
$$
\end{remark}

As expected, we give the angularly-restricted version of the estimates associated to the product of linear solution, $e^{it\Delta}\phi_{N_1}e^{it\Delta}\phi_{N_2}$.

\begin{lemma}\label{lem:bilinear}
Let $0<\theta\ll 1<M$ and $1<N_1\le N_2$.
Suppose that $\phi_{N_1},~\phi_{N_2}\in L^2(\mathbb{R}\times\mathbb{T})$ are the functions with spatial frequencies $N_1,~N_2$ respectively.
Then the space-time function
\begin{equation*}\label{eq:bilinear}
F(t,z)=\int e^{-it(|\zeta_1|^2+|\zeta_2|^2)-iz\cdot(\zeta_1+\zeta_2)}1_{|\xi_1-\xi_2|\gtrsim M}1_{|\cos\angle(\zeta_1,\zeta_2)|\le \theta}\widehat{\phi}_{N_1}(\zeta_1)\widehat{\phi}_{N_2}(\zeta_2)\,(d\zeta_1)(d\zeta_2).
\end{equation*}
has the bound
$$
\|F\|_{L^2(\mathbb{R}\times\mathbb{R}\times\mathbb{T})}\lesssim \frac{\langle \theta N_2\rangle^{1/2}}{M^{1/2}}\|\phi_{N_1}\|_{L^2(\mathbb{R}\times\mathbb{T})}\|\phi_{N_2}\|_{L^2(\mathbb{R}\times\mathbb{T})}.
$$
\end{lemma}

\begin{remark}
In \cite{ckstt1}, such estimate was obtained for the problem in Euclidean space $\mathbb{R}^2$, in which the upper bound on the right-hand side of the estimate is $\theta^{1/2}$.
\end{remark}


\begin{proof}[Proof of Lemma \ref{lem:bilinear}]
The proof follows in a similar way to \cite{ckstt1}.
We first consider the restricted case when the Fourier transform of $\phi_{N_j}$ is supported in the angular sector
$$
\{\zeta_j\mid \arg(\zeta_j)=\ell_j\theta+O(\theta)\}\subset \mathbb{R}\times\mathbb{Z}
$$
where $\ell_j~(j=1,2)$ are arbitrary integers $1\le \ell_j\le 2\pi/\theta$, respectively.
With a truncated setup, we begin by considering the following space-time function
\begin{equation*}\label{eq:bilinear-r}
\begin{split}
F(t,z)= \int &e^{-it(|\zeta_1|^2+|\zeta_2|^2)-iz\cdot(\zeta_1+\zeta_2)} 1_{|\xi_1-\xi_2|\gtrsim M}\widehat{\phi}_{N_1}(\zeta_1)\widehat{\phi}_{N_2}(\zeta_2)\\
 & 1_{\arg(\zeta_1)=\ell_1\theta+O(\theta)}1_{\arg(\zeta_2)=\ell_2\theta+O(\theta)} \,(d\zeta_1)(d\zeta_2).
\end{split}
\end{equation*}
By using Cauchy-Schwarz inequality, we have 
\begin{equation*}
\begin{split}
|\widehat{F}(\tau,\zeta)|^2\lesssim  & \int |\delta(\tau-|\zeta_1|^2-|\zeta-\zeta_1|^2)| |\widehat{\phi}_{N_1}(\zeta_1)|^2|\widehat{\phi}_{N_2}(\zeta-\zeta_1)|^2\,(d\zeta_1)\\
& \sup_{(\tau,\zeta)\in\mathbb{R}\times\mathbb{R}\times\mathbb{Z}}\int 1_{\arg(\zeta_1)=\ell_1\theta+O(\theta)}1_{\arg(\zeta-\zeta_1)=\ell_2\theta+O(\theta)}1_{|\zeta_1|\sim N_1}1_{|\zeta-\zeta_1|\sim N_2}\\
& 1_{|\xi_1-\xi/2|\gtrsim M} \delta(\tau-|\zeta_1|^2-|\zeta-\zeta_1|^2)\,(d\zeta_1).
\end{split}
\end{equation*}
We wish to show
\begin{equation}\label{eq:enough1}
\begin{split}
& \sup_{(\tau,\zeta)\in\mathbb{R}\times\mathbb{R}\times\mathbb{Z}}\int 1_{\arg(\zeta_1)=\ell_1\theta+O(\theta)}1_{\arg(\zeta-\zeta_1)=\ell_2\theta+O(\theta)} 1_{|\zeta_1|\sim N_1}1_{|\zeta-\zeta_1|\sim N_2}1_{|\xi_1-\xi/2|\gtrsim M}\\
&   \delta(\tau-|\zeta_1|^2-|\zeta-\zeta_1|^2)\,(d\zeta_1)\\
\lesssim & \frac{\langle \theta N_2\rangle }{M}.
\end{split}
\end{equation}
With the help of the delta function (density distribution), we may assume $|\tau-|\zeta_1|^2-|\zeta-\zeta_1|^2|\lesssim 1 $ in \eqref{eq:enough1}.
Then the estimate \eqref{eq:enough1} is equivalent to the following
\begin{equation}\label{eq:enough2}
\sup_{(\tau,\xi,\eta)\in\mathbb{R}\times\mathbb{R}\times\mathbb{Z}}|A(\tau,\zeta)|\lesssim \frac{\langle \theta N_2\rangle}{M},
\end{equation}
where
\begin{equation*}
\begin{split}
A(\tau,\zeta)=\{\zeta_1\in\mathbb{R}\times\mathbb{Z}\mid & \arg(\zeta_1)=\ell_1\theta+O(\theta),~\arg(\zeta-\zeta_1)=\ell_2\theta+O(\theta),\\
& |\zeta_1|\sim N_1,~|\zeta-\zeta_1|\sim N_2~|\xi_1-\xi/2|\gtrsim M,\\
& |\zeta_1-\zeta/2|=r+O(1/r)\}
\end{split}
\end{equation*}
with $M\lesssim r\lesssim  N_2$.
Here we will use a Whitney decomposition to the annulus $|\zeta_1-\zeta/2|=r+O(1/r)$ with center $\zeta/2$, which is a partition of the annulus into rectangles such that the size of each rectangle is $1\times O(1/r)$ with central angle $O(1/r)$. 
Now split the angular $\arg(\zeta_1-\zeta)=\ell_2\theta+O(\theta)$ into $\langle \theta r\rangle$-pieces, that is
$$
\arg(\zeta_1-\zeta)\in \ell_2\theta+\left[\frac{k}{2r},\frac{k+1}{2r}\right],\quad |k|\lesssim \langle \theta r\rangle.
$$
Each angular area $\arg(\zeta_1-\zeta)\in \ell_1\theta+[k/r,(k+1)/r]$ intersects with at most three of the angular sectors with central angle $O(1/r)$ induced by the Whitney decomposition to the annulus $|\zeta_1-\zeta/2|=r+O(1/r)$, as we argued above.
Observe that the number of the arcs is at most $O(\langle \theta N_2\rangle)$, which corresponds to the number of $\eta_1$ in the support of the integral on the left-hand side of \eqref{eq:enough1}.
For each $\eta_1$, the measure of $\xi_1$ satisfying $(\xi_1,\eta_1)\in A(\tau,\xi,\eta)$ is at most $O(1/M)$. 
Then we have
\begin{equation*}
|A(\tau,\zeta)| \lesssim \sum_{\eta_1}\frac{1}{M}\lesssim  \frac{\langle \theta N_2\rangle}{M},
\end{equation*}
which shows \eqref{eq:enough2}.

By raising the above estimate, we now show the estimate stated in the lemma.
Following to the instruction of the paper \cite{ckstt1}, we use the angular decomposition technique for computing the Fourier transformation
$$
\phi_{N_j}=\sum_{1\le \ell<2\pi/\theta}\phi_{N_j,\ell},
$$
for $j=1,2$, where the Fourier transform of $\phi_{N_j,\ell}$ is supported in the disjointed angular sector
$$
\{\zeta_j\in \mathbb{R}\times\mathbb{Z}\mid |\zeta_j|\sim N_j,~\arg( \zeta_j)=\theta \ell+O(\theta)\}.
$$
By an argument mimicking of \cite{ckstt1}, the angular restriction $|\cos\angle(\zeta_1,\zeta_2)|\le \theta$ gives that either $|\ell_1-\ell_2|=\pi/(2\theta)+O(1)$ or $|\ell_1-\ell_2|=3\pi/(2\theta)+O(1)$ holds.
Then
\begin{equation}
\begin{split}
\|F\|_{L^2(\mathbb{R}\times\mathbb{R}\times\mathbb{T})} & \lesssim 
\sum_{\scriptstyle |\ell_1-\ell_2|=\pi/(2\theta)+O(1) \atop{\scriptstyle\mbox{or}~|\ell_1-\ell_2|=3\pi/(2\theta)+O(1)}}\frac{\langle \theta N_2\rangle^{1/2}}{M^{1/2}}\|\phi_{N_1,\ell_1}\|_{L^2(\mathbb{R}\times\mathbb{T})}\|\phi_{N_2,\ell_2}\|_{L^2(\mathbb{R}\times\mathbb{T})}\\
& \lesssim \frac{\langle \theta N_2\rangle^{1/2}}{M^{1/2}}\left(\sum_{1\le \ell_1<2\pi/\theta}\|\phi_{N_1,\ell_1}\|_{L^2(\mathbb{R}\times\mathbb{T})}^2\right)\left(\sum_{1\le \ell_2<2\pi/\theta}\|\phi_{N_2,\ell_2}\|_{L^2(\mathbb{R}\times\mathbb{T})}^2\right)\\
& \lesssim \frac{\langle \theta N_2\rangle^{1/2}}{M^{1/2}}\|\phi_{N_1}\|_{L^2(\mathbb{R}\times\mathbb{T})}\|\phi_{N_2(\mathbb{R}\times\mathbb{T})}\|_{L^2}.
\end{split}
\end{equation}
\end{proof}

The next proposition is a direct consequence of Lemma \ref{lem:bilinear}.

\begin{proposition}\label{prop:bilinearX}
Let $0<\theta\ll 1<M$ and $1<N_1\le N_2$.
Suppose that $f_{N_1},~f_{N_2}\in X^{0,1/2+}$ are the functions with spatial frequencies $N_1,~N_2$ respectively.
Then the space-time function
\begin{equation*}
\begin{split}
F(t,z)=\int_{\mathbb{R}^2}\!\int  & e^{-it(\tau_1+\tau_2)-iz\cdot(\zeta_1+\zeta_2)}1_{|\xi_1-\xi_2|\gtrsim M}1_{|\cos\angle(\zeta_1,\zeta_2)|\le \theta}\\
 & \widehat{f}_{N_1}(\tau_1,\zeta_1)\widehat{f}_{N_2}(\tau_2,\zeta_2)\,(d\zeta_1)(d\zeta_2)d\tau_1d\tau_2.
\end{split}
\end{equation*}
has the bound
$$
\|F\|_{L^2(\mathbb{R}\times\mathbb{R}\times\mathbb{T})}\lesssim \frac{\langle \theta N_2\rangle^{1/2} }{M^{1/2}}\|f_{N_1}\|_{X^{0,1/2+}}\|f_{N_2}\|_{X^{0,1/2+}}.
$$
\end{proposition}

Our next lemma concerns with the estimates associated with  the product of linear solution, $e^{it\Delta}\phi_{N_1}\overline{e^{it\Delta}\phi_{N_2}}$.

\begin{lemma}\label{lem:bilinear0}
Let $0<\theta\ll 1<M,~\ell\in\mathbb{Z}$ and $N_1,~N_2>1$.
Suppose that $\phi_{N_1},~\phi_{N_2,\ell}\in L^2(\mathbb{R}\times\mathbb{T})$ are the functions with spatial frequencies $N_1,~N_2$ respectively.
Also assume that the Fourier transform of $\phi_{N_2,\ell}$ is supported on $\{\zeta\in\mathbb{R}\times\mathbb{Z}\mid \arg(\zeta)=\ell\theta+O(\theta)\}$.
Then the space-time function
\begin{equation*}\label{eq:bilinear}
\begin{split}
F(t,z) \int &  e^{-it(|\zeta_1|^2-|\zeta_2|^2)-iz\cdot(\zeta_1+\zeta_2)}1_{|\xi_1+\xi_2|\gtrsim  M}1_{\arg(\zeta_2)=\ell\theta+O(\theta)}\\
& \widehat{\phi}_{N_1}(\zeta_1)\widehat{\phi}_{N_2,\ell}(\zeta_2)\,(d\zeta_1)(d\zeta_2)
\end{split}
\end{equation*}
has the bound
$$
\|F\|_{L^2(\mathbb{R}\times\mathbb{R}\times\mathbb{T})}\lesssim \frac{\langle \theta (N_1+N_2)\rangle^{1/2}}{M^{1/2}}\|\phi_{N_1}\|_{L^2(\mathbb{R}\times\mathbb{T})}\|\phi_{N_2,\ell}\|_{L^2(\mathbb{R}\times\mathbb{T})},
$$
where the constant on the right-hand side is independent of $\theta,~M,~\ell,~N_1,~N_2$. 
\end{lemma}

\begin{proof}
We will use the similar argument to the proof of Lemma \ref{lem:bilinear}.
It suffices to show
\begin{equation}\label{eq:+}
\begin{split}
& \sup_{(\tau,\zeta)\in\mathbb{R}\times\mathbb{R}\times\mathbb{T}}\int 1_{\arg(\zeta_2)=\ell\theta+O(\theta)}1_{|\xi_1+\xi_2|\gtrsim M}1_{|\zeta-\zeta_1|\sim N_2}
  \delta(\tau-|\zeta_1|^2+|\zeta-\zeta_1|^2)\,(d\zeta_1)\\
\lesssim & \frac{\langle \theta N_2\rangle }{M}.
\end{split}
\end{equation}
We again divide the arc $|\zeta_1-\zeta|=r\lesssim N_1+N_2,~\arg(\zeta_1-\zeta)=\ell \theta+O(\theta)$ into the arcs of length $O(1)$ with angler $O(1/\langle r\rangle)$.
The number of the arcs is at most $O(\langle \theta (N_1+N_2)\rangle )$. 
Since $-|\zeta_1|^2+|\zeta-\zeta_1|^2=|\zeta|^2-2\xi\xi_1-2\eta\eta_1$ and $|\xi|\sim M$, we have that the integral on the left-hand side of \eqref{eq:+} has the bound
$$
c\frac{\langle \theta (N_1+N_2)\rangle}{M},
$$
which completes the proof.
\end{proof}

By Lemma \ref{lem:bilinear0}, we obtain the following proposition.

\begin{proposition}\label{prop:bilinear0}
Let $0<\theta\ll 1<M,~\ell\in\mathbb{Z}$ and $N_1,~N_2>1$.
Suppose that $f_{N_1},~f_{N_2,\ell}\in L^2(\mathbb{R}\times\mathbb{R}\times\mathbb{T})$ are the functions with spatial frequencies $N_1,~N_2$ respectively, and that the Fourier transform of $f_{N_2,\ell}$ is supported on $\{(\tau,\zeta)\in\mathbb{R}\times\mathbb{R}\times\mathbb{Z}\mid  \arg(\zeta)=\ell\theta+O(\theta)\}$.
Then the space-time function
\begin{equation*}
\begin{split}
F(t,z)=\int_{\mathbb{R}^2}\!\int  & e^{-it(\tau_1+\tau_2)-iz\cdot(\zeta_1+\zeta_2)} 1_{|\xi+\xi_2|\sim N_1}1_{\arg(\zeta_2)=\ell\theta+O(\theta)}\\
& \widehat{f_{N_1}}(\tau_1,\zeta_1)\widehat{\overline{f_{N_2,\ell}}}(\tau_2,\zeta_2)\,(d\zeta_1)(d\zeta_2)d\tau_1d\tau_2.
\end{split}
\end{equation*}
has the bound
$$
\|F\|_{L^2(\mathbb{R}\times\mathbb{R}\times\mathbb{T})}\lesssim \frac{\langle \theta (N_1+N_2)\rangle^{1/2}}{M^{1/2}}\|f_{N_1}\|_{X^{0,1/2+}}\|f_{N_2,\ell}\|_{X^{0,1/2+}},
$$
where the constant on the right-hand side is independent of $\theta,~M,~\ell,~N_1,~N_2$. 
\end{proposition}

Finally we give the observation for the case of $N_1=N_2=N_0$.
Let us use the formula $|\zeta_1|^2+|\zeta-\zeta_1|^2=2|\zeta_1-\zeta/2|^2+|\zeta|^2/2$.
For each $(\tau,\zeta)\in\mathbb{R}\times\mathbb{R}\times\mathbb{Z}$ satisfying that $\tau-|\zeta|^2/2\le 100$, we have
$$
\int 1_{\{\zeta_1\in\mathbb{R}\times\mathbb{Z}\mid |\zeta_1|\sim N_0,~|\zeta-\zeta_1|\sim N_0,~|\tau-|\zeta_1|^2-|\zeta-\zeta_1|^2|\le  1\}}\,(d\zeta_1)\lesssim 1.
$$
If $(\tau,\zeta)\in\mathbb{R}\times\mathbb{R}\times\mathbb{Z}$ satisfies $\tau-|\zeta|^2/2|\ge 100$, we can write
\begin{equation*}
\begin{split}
& \{\zeta_1\in\mathbb{R}\times\mathbb{Z}\mid |\zeta_1|\sim N_0,~|\zeta-\zeta_1|\sim N_0,~|\tau-|\zeta_1|^2-|\zeta-\zeta_1|^2|\le 1\}\\
&\subset
\bigcup_{|\eta_1|=0}^{O\left(\sqrt{\tau-|\zeta|^2/2}\right)}\{\xi_1\in\mathbb{R}\mid \zeta_1=(\xi_1,\eta_1)\in\mathbb{R}\times\mathbb{Z},~
|2|\zeta_1-\zeta/2|^2-\tau+|\zeta|^2/2|\le 1\}.
\end{split}
\end{equation*}
Thus
\begin{equation*}
\begin{split}
& 
\left|\{\zeta_1\in\mathbb{R}\times\mathbb{Z}\mid |\zeta_1|\sim N_0,~|\zeta-\zeta_1|\sim N_0,~|\tau-|\zeta_1|^2-|\zeta-\zeta_1|^2|\ll 1\}\right|\\
\lesssim & \sum_{\eta_1=0}^{O\left(\sqrt{\tau-|\zeta|^2/2}\right)}\left(\sqrt{(\sqrt{\tau-|\zeta|^2/2+1})^2-\eta_1^2}-\sqrt{(\sqrt{\tau-|\zeta|^2/2-1})^2-\eta_1^2}\right) \\
\lesssim & \left(\tau-|\zeta|^2/2\right)\sum_{\eta_1=0}^{O\left(\sqrt{\tau-|\zeta|^2/2}\right)}\frac{1}{\sqrt{(\tau-|\zeta|^2/2)^2-\eta_1^2}}\\
\lesssim & \int_0^1\frac{dt}{\sqrt{1-t^2}}\lesssim 1.
\end{split}
\end{equation*}
Collecting these estimates, we arrive at the following
\begin{equation}\label{eq:N_1N_2}
\sup_{(\tau,\zeta)\in\mathbb{R}\times\mathbb{R}\times\mathbb{Z}}\int 1_{\{\zeta_1\in\mathbb{R}\times\mathbb{Z}\mid |\zeta_1|\sim N_0,~|\zeta-\zeta_1|\sim N_0,~|\tau-|\zeta_1|^2-|\zeta-\zeta_1|^2|\le 1\}}\,(d\zeta_1)\lesssim 1.
\end{equation}
With this nice estimate, we revisit the proof of Lemma \ref{lem:bilinear}.
By replacing \eqref{eq:enough1} with \eqref{eq:N_1N_2}, it follows that for any $\phi_{N_0}\in L^2(\mathbb{R}\times\mathbb{T})$ with spatial frequencies $N_0$
\begin{equation}\label{eq:N_1=N_2}
\|e^{-it\Delta}\phi_{N_0}\|_{L^4(\mathbb{R}\times\mathbb{R}\times\mathbb{T})}\lesssim \|\phi_{N_0}\|_{L^2(\mathbb{R}\times\mathbb{T})}.
\end{equation}
As dealt in Proposition \ref{prop:bilinearX}, we immediately obtain the following proposition.
\begin{proposition}
For a space-time function $f$, we have
\begin{equation}\label{eq:4}
\|f\|_{L^{4}(\mathbb{R}\times\mathbb{R}\times\mathbb{T})}\lesssim \|f\|_{X^{0+,1/2+}}
\end{equation}
and
\begin{equation}\label{eq:4+}
\|f\|_{L^{4+}(\mathbb{R}\times\mathbb{R}\times\mathbb{T})}\lesssim \|f\|_{X^{0+,1/2+}}.
\end{equation}
Moreover, for any $f_{N_0}$ with spatial frequencies $N_0$
\begin{equation}\label{eq:4-}
\|f_{N_0}\|_{L^{4}(\mathbb{R}\times\mathbb{R}_0\times\mathbb{T})}\lesssim \|f_{N_0}\|_{X^{0,1/2+}}
\end{equation}
and
\begin{equation}\label{eq:4--}
\|f_{N_0}\|_{L^{4-}(\mathbb{R}\times\mathbb{R}\times\mathbb{T})}\lesssim \|f_{N_0}\|_{X^{0,1/2-}}.
\end{equation}
\end{proposition}

\begin{proof}
The estimate \eqref{eq:N_1=N_2} implies \eqref{eq:4}.
Interpolating \eqref{eq:4} and the classical Sobolev embedding
$$
\|f\|_{L^{\infty}(\mathbb{R}\times\mathbb{R}\times\mathbb{T})}\lesssim \|f\|_{X^{1+,1/2+}},
$$
we obtain \eqref{eq:4+}.
The estimate \eqref{eq:4-} follows from \eqref{eq:N_1=N_2}.
The estimate \eqref{eq:4--} follows from interpolation \eqref{eq:4-} and the trivial estimate $\|f_{N_0}\|_{L^2(\mathbb{R}\times\mathbb{R}\times\mathbb{T})}=\|f_{N_0}\|_{X^{0,0}}$.
\end{proof}

\section{Reduction of the energy}\label{sec:reductionenergy}

In this section, we modify the energy function via normal form reduction argument.
To get a better energy function associated to $H^s$ solution, we make use of a spatial Fourier multiplier operator $I:H^s(\mathbb{R}\times\mathbb{T})\to H^1(\mathbb{R}\times\mathbb{T})$ by a smooth monotone multiplier $m:\mathbb{R}\times\mathbb{Z}\to (0,\infty)$ satisfying 
$$
m(\zeta)=
\begin{cases}
1, & |\zeta|< N,\\
\left(\frac{|\zeta|}{N}\right)^{s-1}, & |\zeta|>2N,
\end{cases}
$$
where $N>1$ is to be determined later, that is, in Section \ref{sec:proofTheorem}. 
We point out that the operator $I$ maps from $H^s(\mathbb{R}\times\mathbb{T})$ onto $H^1(\mathbb{R}\times\mathbb{T})$ as it can be seen in the following estimate.
For $s>1$,
\begin{equation}\label{eq:HIH}
\|f\|_{H^s(\mathbb{R}\times\mathbb{T})}\lesssim N^{s-1}\|If\|_{H^1(\mathbb{R}\times\mathbb{T})}\lesssim N^{s-1}\|f\|_{H^s(\mathbb{R}\times\mathbb{T})}.
\end{equation}

It is convenient to introduce some notation for the multilinear expressions.
If $k$ is an integer, we write
$$
\int M(\zeta_1,\ldots,\zeta_k)1_{\zeta_1+\cdots +\zeta_k=0}(d\zeta_1)\ldots (d\zeta_k)=
\int_*  M(\zeta_1,\ldots,\zeta_k).
$$
Given the multiplier $I$, we provide the definition of the modified energy as 
\begin{equation}\label{eq:modified-e}
\begin{split}
E_I[u](t)=& \frac{1}{2}\int_* |\zeta_1||\zeta_2|m(\zeta_1)m(\zeta_2)\widehat{u}(t,\zeta_1)\widehat{\overline{u}}(t,\zeta_2)\\
& +\frac{1}{4(2\pi)^2}\int_* \Lambda_4(\zeta_1,\zeta_2,\zeta_3,\zeta_4) \widehat{u}(t,\zeta_1)\widehat{\overline{u}}(t,\zeta_2)\widehat{u}(t,\zeta_3)\widehat{\overline{u}}(t,\zeta_4),
\end{split}
\end{equation}
with
\begin{equation*}
\begin{split}
\Lambda_4(\zeta_1,\zeta_2,\zeta_3,\zeta_4)=& \frac{m(\zeta_1)^2|\zeta_1|^2-m(\zeta_2)^2|\zeta_2|^2+m(\zeta_3)^2|\zeta_3|^2-m(\zeta_4)^2|\zeta_4|^2}{|\zeta_1|^2-|\zeta_2|^2+|\zeta_3|^2-|\zeta_4|^2}\\
& \times 1_{|\cos\angle(\zeta_{12},\zeta_{14})|>\theta_0~\mathrm{or}~\max\{|\zeta_j|\mid 1\le j\le 4\}\ll N}(\zeta_1,\zeta_2,\zeta_3,\zeta_4)
\end{split}
\end{equation*}
and
$$
\theta_0=\theta_0(\zeta_1,\zeta_2,\zeta_3,\zeta_4)=\frac{1}{1+|\zeta_1|+|\zeta_2|+|\zeta_3|+|\zeta_4|}.
$$
In the above, we used the notations $\zeta_{ij}=\zeta_i+\zeta_j$.

The following two lemmas state some elementary inequalities for the function of $\Lambda_4$. 

\begin{lemma}\label{lem:m}
Let $\zeta_j\in\mathbb{R}\times\mathbb{Z}~(1\le j\le 4)$ with $|\zeta_1|\gtrsim \max\{|\zeta_j|\mid 2\le j\le 4\}$ with $\zeta_1+\zeta_2+\zeta_3+\zeta_4=0$.
Then
\begin{equation*}
\begin{split}
& |m(\zeta_1)^2|\zeta_1|^2-m(\zeta_2)^2|\zeta_2|^2+m(\zeta_3)^2|\zeta_3|^2-m(\zeta_4)^2|\zeta_4|^2|\\
 \lesssim &
m(\zeta_1)^2\times
\begin{cases}
|\zeta_1|^2, & \min\{|\zeta_{12}|,|\zeta_{14}|\}\gtrsim |\zeta_1|,\\
|\zeta_{12}||\zeta_{14}|, & \min\{|\zeta_{12}|,|\zeta_{14}|\}\ll |\zeta_1|.
\end{cases}
\end{split}
\end{equation*}
\end{lemma}

\begin{proof}
We repeat the proof of Lemma 4.1 in \cite{ckstt2}.
Since the function $m(\zeta)^2|\zeta|^2$ is monotone increasing, we can have
$$
\left|m(\zeta_1)^2|\zeta_1|^2-m(\zeta_2)^2|\zeta_2|^2+m(\zeta_3)^2|\zeta_3|^2-m(\zeta_4)^2|\zeta_4|^2\right|\lesssim 
m(\zeta_1)^2|\zeta_1|^2.
$$
Hence, we reduce the problem to the case $\min\{|\zeta_{12}|,|\zeta_{14}|\}\ll |\zeta_1|$ in the following.
 
If $\max\{|\zeta_{12}|,|\zeta_{14}|\}\ll |\zeta_1|$, we proceed as
\begin{equation*}
\begin{split}
& \left|m(\zeta_1)^2|\zeta_1|^2-m(\zeta_2)^2|\zeta_2|^2+m(\zeta_3)^2|\zeta_3|^2-m(\zeta_4)^2|\zeta_4|^2\right|\\
= & \left|\int_0^1\!\!\int_0^1(\zeta_{12}\cdot \nabla)(\zeta_{14}\cdot\nabla)m(\zeta_1-s\zeta_{12}-t\zeta_{14})^2|\zeta_1-s\zeta_{12}-t\zeta_{14}|^2\,dsdt\right|\\
\lesssim  &m(\zeta_1)^2|\zeta_{12}||\zeta_{14}|
\end{split}
\end{equation*}

Finally we observe the case when $\min\{|\zeta_{12}|,|\zeta_{14}|\}\ll |\zeta_1|\lesssim  \max\{|\zeta_{12}|,|\zeta_{14}|\}$.
By symmetry, we suppose $|\zeta_{14}|\ll |\zeta_1|\sim |\zeta_{12}|$.
Also we may suppose $|\zeta_2|\gg |\zeta_{23}|$, since the proof for the case $|\zeta_2|\lesssim  |\zeta_{23}|$ follows similarly.
In this case, we see
\begin{equation*}
\begin{split}
& \left|m(\zeta_1)^2|\zeta_1|^2-m(\zeta_2)^2|\zeta_2|^2+m(\zeta_3)^2|\zeta_3|^2-m(\zeta_4)^2|\zeta_4|^2\right|\\
= & \left|\int_0^1(\zeta_{14}\cdot\nabla)\left( m(\zeta_1-t\zeta_{14})^2|\zeta_1-t\zeta_{14}|^2+ m(\zeta_2-t\zeta_{23})^2|\zeta_2-t\zeta_{23}|^2\right)\,dt\right|\\
\lesssim & m(\zeta_1)^2|\zeta_1||\zeta_{14}|\\
\sim& m(\zeta_1)^2|\zeta_{12}||\zeta_{14}|.
\end{split}
\end{equation*}
This concludes the proof of the lemma.
\end{proof}

Next, we prove the following preliminary lemma, which is due to the estimate for the second term on the right-hand side of  \eqref{eq:modified-e}.

\begin{lemma}\label{lem:error}
We have
\begin{equation*}
\left|E_I[f]-E[If]\right|\lesssim \frac{1}{N^{1-}}\|If\|_{H^1(\mathbb{R}\times\mathbb{T})}^4.
\end{equation*}
\end{lemma}

\begin{proof}
Recall that
$$
\|\nabla If\|_{L^2(\mathbb{R}\times\mathbb{T})}^2= \int_* |\zeta_1||\zeta_2|m(\zeta_1)m(\zeta_2)\widehat{f}(\zeta_1)\widehat{\overline{f}}(\zeta_2).
$$
Then it suffices to show
\begin{equation}\label{eq:Lambda_4-minus}
\begin{split}
& \left|\int_* \left(\Lambda_4(\zeta_1,\zeta_2,\zeta_3,\zeta_4) -m(\zeta_1)m(\zeta_2)m(\zeta_3)m(\zeta_4)\right) \widehat{f}(\zeta_1)\widehat{\overline{f}}(\zeta_2)\widehat{f}(\zeta_3)\widehat{\overline{f}}(\zeta_4)\right|\\
 \lesssim & \frac{1}{N^{1-}}\|If\|_{H^1(\mathbb{R}\times\mathbb{T})}^4.
\end{split}
\end{equation}
By a Littlewood-Paley decomposition, we restrict each function to a dyadic frequency band $|\zeta_j|\sim N_j$ and will sum in the $N_j$ at the end of the argument.
In this content, we write $f_{N_j}$ to denote the function of a dyadic frequency band $|\zeta|\sim N_j$. 
Without of loss generality, we may assume $N_1=\max\{N_j\mid 1\le j\le 4\}\sim \max\{N_j\mid 2\le j\le 4\}$.

In the case when $N_1\ll N$, we can write $\Lambda_4(\zeta_1,\zeta_2,\zeta_3,\zeta_4)=1$.
Therefore we assume $N_1\gtrsim N$.
%
By Lemma \ref{lem:m}, we note the following piecewise estimate
$$
|\Lambda_4(\zeta_1,\zeta_2,\zeta_3,\zeta_4)|\lesssim \frac{m(N_1)^2}{\theta_0}.
$$
For simplicity, we assume $N_1\sim N_2\gtrsim N+N_3+N_4+N_5$ .
Then we take $L^2(\mathbb{R}\times\mathbb{T})$ for $f_{N_j}~(j=1,2)$ and $L^{\infty}(\mathbb{R}\times\mathbb{T})\hookleftarrow H^{1+}(\mathbb{R}\times\mathbb{T})$ for $f_{N_j}~(j=3,4)$ to have that the contribution of this case to the left-hand side of \eqref{eq:Lambda_4-minus} is bounded by
\begin{equation*}
\begin{split}
& c\sum_{\scriptstyle N_1\sim N_2\gtrsim N+N_3+N_4} \left(\frac{m(N_1)^2}{\theta_0}+m(N_1)m(N_2)m(N_3)m(N_4)\right)\frac{1}{N_1N_2}\prod_{j=1}^4\|f_{N_j}\|_{H^1(\mathbb{R}\times\mathbb{T})}\\
  \lesssim  &  \frac{1}{N^{1-}}\|If\|_{H^1(\mathbb{R}\times\mathbb{T})}^4.
\end{split}
\end{equation*}
%
Hence we prove the desired estimate in \eqref{eq:Lambda_4-minus}.
\end{proof}

\section{Proof of Theorem \ref{thm:growth}}\label{sec:proofTheorem}

This section devotes to the proof of Theorem \ref{thm:growth}.
First we begin by summarizing the linear estimates adapted in the space $X^{s,b}$.
The proof of this lemma can be found in \cite{b1,kpv}.

\begin{lemma}\label{lem:linearX}
Let $I=[0,\delta]\subset \mathbb{R}$ be a bounded time interval.
For $s\in\mathbb{R}$ and $0<b'<b<1/2$, there exist constants $C>0$ and $c=c(b,b')>0$ such that
\begin{equation*}
\|e^{-it\Delta}\phi\|_{X^{s,1/2+}(I)}\le C \|\phi\|_{H^s(\mathbb{R}\times\mathbb{T})},
\end{equation*}
\begin{equation*}
\left\|\int_0^te^{-i(t-t')\Delta}F(\cdot,t')\,dt'\right\|_{X^{s,1/2+}}\le C \|F\|_{X^{s,-1/2+}(I)},
\end{equation*}
\begin{equation*}
\|F||_{X^{s,-b}(I)}\le C \delta^{c}\|F\|_{X^{s,-b'}(I)}.
\end{equation*}
\end{lemma}

The following local well-posedness result holds in the space $IH^s$.

\begin{theorem}[Modified local well-posedness]\label{thm:LWP}
Let $s>1$, and let $u_0\in H^s(\mathbb{R}\times\mathbb{T})$.
Then the Cauchy problem \eqref{eq:NLS} is locally well-posed in $H^s(\mathbb{R}\times\mathbb{T})$.
Moreover, the solution exists on the time interval $[0,\delta_0]$ with the lifetime
$$
\delta_0\sim \|Iu_0\|_{H^1(\mathbb{R}\times\mathbb{T})}^{-c},
$$
and the solution $u(t)$ satisfies the estimate
$$
\|Iu\|_{X^{1,1/2+}([0,\delta_0])}\lesssim \|Iu_0\|_{H^1(\mathbb{R}\times\mathbb{T})}.
$$
\end{theorem}

\begin{proof}
The problem of existence of solutions is based on the fixed point argument in $IX^{s,1/2+}([0,\delta_0])$ for some $\delta_0>0$.
Solving the Cauchy problem \eqref{eq:NLS}, we consider the following Duhamel formulation of the problem as
$$
Iu(t)=e^{it\Delta}Iu_0-i\int_0^te^{i(t-t')\Delta}[I(|u|^2u)](t')\,dt'.
$$
Moreover, letting $v=Iu$ and denoting
$$
T[v](t)=e^{it\Delta}Iu_0-i\int_0^te^{i(t-t')\Delta}[I(|I^{-1}v|^2I^{-1}v)](t')\,dt',
$$
we are looking for a fixed point $v(t)$ in the space $X^{1,1/2+}([0,\delta_0])$ for the mapping $T$ such that $Tv=v$.

By combining Lemma \ref{lem:linearX} with \eqref{eq:HIH}, we obtain the following estimates
\begin{equation}\label{eq:X_linear}
\|e^{it\Delta}Iu_0\|_{X^{1,1/2+}([0,\delta_0])}\lesssim \|Iu_0\|_{H^1(\mathbb{R}\times\mathbb{T})}\lesssim \|u_0\|_{H^s(\mathbb{R}\times\mathbb{T})}
\end{equation}
and
\begin{equation}\label{eq:Duhamel}
\begin{split}
\left\|\int_0^te^{i(t-t')\Delta}[I(|I^{-1}v|^2I^{-1}v)](t')\,dt'\right\|_{X^{1,1/2+\varepsilon}([0,\delta_0])} & \lesssim \delta_0^c\|I(|I^{-1}v|^2I^{-1}v)\|_{X^{1,-1/2+3\varepsilon/2}([0,\delta_0])}.
\end{split}
\end{equation}
For the second term, we wish to show
$$
\|I(|I^{-1}v|^2I^{-1}v)\|_{X^{1,-1/2+3\varepsilon/2}([0,\delta_0])}\lesssim \|v\|_{X^{1,1/2+\varepsilon}([0,\delta_0])}^3.
$$
If we replace $v(t)$ by $\eta(t)v(t)$ where $\eta$ is a smooth bump function with compact support, then above estimate can ce reduced to the following
\begin{equation*}
\|I(I^{-1}v_1\overline{I^{-1}v_2}I^{-1}v_3)\|_{X^{1,-1/2+3\varepsilon/2}}\lesssim \prod_{j=1}^3\|v_j\|_{X^{1,1/2+\varepsilon}}.
\end{equation*}
By duality argument, we recast this estimate as
\begin{equation}\label{eq:duality}
\begin{split}
& \int_{\tau_1+\tau_2+\tau_3+\tau_4=0}\int_* \frac{\langle\zeta_4\rangle m(\zeta_4)\langle \tau_4+|\zeta_4|^2\rangle^{3\varepsilon/2}}{\langle \zeta_1\rangle m(\zeta_1)\langle \zeta_2\rangle m(\zeta_2)\langle \zeta_3\rangle m(\zeta_3)}\widehat{f_1}(\tau_1,\zeta_1)\widehat{\overline{f_2}}(\tau_2,\zeta_2)\widehat{f_3}(\tau_3,\zeta_2)\widehat{\overline{f_4}}(\tau_4,\zeta_4)\\
\lesssim & \prod_{j=1}^4\|f_j\|_{X^{0,1/2+\varepsilon}},
\end{split}
\end{equation}
for positive functions $\widehat{f_1},~\widehat{\overline{f_2}},~\widehat{f_3},~\widehat{\overline{f_4}}$.
The identities
$$
\zeta_1+\zeta_2+\zeta_3+\zeta_4=0
$$
and
$$
\tau_4+|\zeta_4|^2=-(\tau_1-|\zeta_1|^2)-(\tau_2+|\zeta_2|^2)-(\tau_4-|\zeta_4|^2)-2\zeta_{12}\cdot\zeta_{14}
$$
imply
$$
|\tau_4+|\zeta_4|^2|\lesssim \max\left\{|\tau_1-|\zeta_1|^2|,|\tau_2+|\zeta_2|^2|,|\tau_3-|\zeta_3|^2|,\max_{1\le j\le 3}|\zeta_j|^2\right\}
$$
and 
$$
\langle\zeta_4\rangle m(\zeta_4)\lesssim \max\{\langle\zeta_j\rangle m(\zeta_j)\mid 1\le j\le 3\}.
$$
By a Littlewood-Paley decomposition, we restrict each $f_{j}$ to a dyadic frequency band $|\zeta_j|\sim N_j$ by writing $f_{j,N_j}$ and will sum in the $N_j$ at the end of the argument, as used before.
We split the case into two sub-cases; $N_4 \sim \max\{N_j\mid 1\le j\le 3\}$, $N_4\ll \max\{N_j\mid 1\le j\le 3\}$, and develop the proof with the case by case analysis.

In the case when $N_4 \sim \max\{N_j\mid 1\le j\le 3\}$, we may assume $N_4\sim N_1=\max\{N_j\mid 1\le j\le 3\}$ by symmetry.
Apply the $L^{4-}(\mathbb{R}\times\mathbb{R}\times\mathbb{T})$ Strichartz estimate \eqref{eq:4--} to $f_{N_j}$ of $j=1,4$ and the $L^{4+}(\mathbb{R}\times\mathbb{R}\times\mathbb{T})$ Strichartz estimate \eqref{eq:4+} to $f_{N_j}$ of $j=2,3$.
Then the contribution of this case to the left-hand side of \eqref{eq:duality} is bounded by
\begin{equation*}
\begin{split}
& c\sum_{N_1\sim N_4> N_2+N_3}\|f_{1,N_1}\|_{X^{0,1/2-}}\|f_{2,N_2}\|_{X^{-1+,1/2+}}\|f_{3,N_3}\|_{X^{-1+,1/2+}}\|f_{4,N_4}\|_{X^{0,1/2+3\varepsilon/2-}}\\
\lesssim & \prod_{j=1}^4\|f_j\|_{X^{0,1/2+\varepsilon}}.
\end{split}
\end{equation*}

Next we deal with the case when $N_4\ll \max\{N_j\mid 1\le j\le 3\}$.
By symmetry, it is enough to assume $N_1\sim N_2\gtrsim \max\{N_j\mid 1\le j\le 3\}\gg N_4$.
Same as above, we get that the contribution of this case to the left-hand side of \eqref{eq:duality} is bounded by
\begin{equation*}
\begin{split}
& c\sum_{\scriptstyle N_1\sim N_2\gg N_4 \atop{\scriptstyle N_1\gtrsim N_3}}\|f_{1,N_1}\|_{X^{-1/2+,1/2-}}\|f_{2,N_2}\|_{X^{-1/2+,1/2+}}\|f_{3,N_3}\|_{X^{-1+,1/2+}}\|f_{4,N_4}\|_{X^{0,1/2+3\varepsilon/2-}}\\
\lesssim & \prod_{j=1}^4\|f_j\|_{X^{0,1/2+\varepsilon}}.
\end{split}
\end{equation*}
Putting all these estimates together, therefore we obtain \eqref{eq:duality}; thus by \eqref{eq:Duhamel}
\begin{equation}\label{eq:X_Duhamel}
\left\|\int_0^te^{i(t-t')\Delta}[I(|I^{-1}v|^2I^{-1}v)](t')\,dt'\right\|_{X^{1,1/2+}([0,\delta_0])} \lesssim \delta_0^c \|v\|_{X^{1,1/2+}([0,\delta_0])}^3.
\end{equation}

Let now consider the closed ball $B$ centered at the origin
$$
B=\left\{v\in X^{1,1/2+}([0,\delta_0])\mid \|v\|_{X^{1,1/2+}([0,\delta_0])}\le C\|Iu_0\|_{H^1(\mathbb{R}\times\mathbb{T})}\right\}.
$$
By \eqref{eq:X_linear} and \eqref{eq:X_Duhamel}, we have
$$
\|T[v]\|_{X^{1,1/2+}([0,\delta_0])}\le c\|Iu_0\|_{H^1(\mathbb{R}\times\mathbb{T})}+c\delta_0^c\|v\|_{X^{1,1/2+}([0,\delta_0])}^3\le C\|u_0\|_{H^s}
$$
for selecting $\delta_0>0$ appropriately smaller such that $\delta_0^c\|u_0\|_{H^s}^2\ll 1$, that is $T[v]\in B$ for $v\in B$.
Moreover, a similar argument applies to $T[v_1]-T[v_2]$, thus proving
$$
\|T[v_1]-T[v_2]\|_{X^{1,1/2+}([0,\delta_0])}\le \frac{1}{2}\|v_1-v_2\|_{X^{1,1/2+}([0,\delta_0])}
$$
for $v_1,v_2\in B$.
Then we have shown that the map $T$ is a contraction on $B$.
Therefore by Banach's fixed point theorem, we conclude that there exists a unique solution $u(t)=I^{-1}v(t)$ of \eqref{eq:NLS}, which is our local well-posedness result in the theorem.
\end{proof}

\begin{remark}
In Theorem \ref{thm:LWP}, the time $\delta_0$ of existence of unique solution to \eqref{eq:NLS} depends only on $H^1$ norm of data.
Hence the global well-posedness result holds in $H^s(\mathbb{R}\times\mathbb{T})$ for $s\ge 1$, which can be found in the paper \cite{dy}. 
\end{remark}

We provide here the estimate on $m(\zeta)$ with the angularly restriction.

\begin{lemma}\label{lem:m1}
Let $\zeta_j\in\mathbb{R}\times\mathbb{Z}~(1\le j\le 4)$ such that $\zeta_1+\zeta_2+\zeta_3+\zeta_4=0$.
Assume $|\zeta_1|\sim |\zeta_2|\gtrsim |\zeta_3|\gtrsim  |\zeta_4| + 1$.
Then
\begin{equation*}
\begin{split}
& \left|m(\zeta_1)^2|\zeta_1|^2-m(\zeta_2)^2|\zeta_2|^2+m(\zeta_3)^2|\zeta_3|^2-m(\zeta_4)^2|\zeta_4|^2\right|\\
\lesssim & m(\zeta_1)^2|\zeta_{12}|  \left(|\zeta_1||\cos\angle(\zeta_{12},\zeta_{14})|+|\zeta_3|\right).
\end{split}
\end{equation*}
\end{lemma}

\begin{proof}
We have shown the similar estimate if $m(\zeta)$ is decreasing function in the paper \cite{ckstt1} dealing the case $s<1$.
In the following, we revisit that proof.

In particular, when $|\zeta_1|\sim |\zeta_3|$, the proof follows immediately.
Clearly, $|m(\zeta_3)^2|\zeta_3|^2-m(\zeta_4)^2|\zeta_4|^2|\lesssim m(\zeta_1)^2|\zeta_{34}||\zeta_3|$.
Then it suffices to show
$$
\left|m(\zeta_1)^2|\zeta_1|^2-m(\zeta_2)^2|\zeta_2|^2\right |\lesssim m(\zeta_1)^2|\zeta_1||\zeta_{12}||\cos\angle(\zeta_{12},\zeta_{14})|
$$
for $|\zeta_1|\sim |\zeta_2|\gg |\zeta_3|\gtrsim |\zeta_4|>1$.
By the identity
\begin{equation*}
\begin{split}
|\zeta_1|^2-|\zeta_2|^2+|\zeta_3|^2-|\zeta_4|^2= & 2(\zeta_1+\zeta_2)\cdot(\zeta_1+\zeta_4)\\
= & 2|\zeta_{12}||\zeta_{14}|\cos\angle(\zeta_{12},\zeta_{14}),
\end{split}
\end{equation*}
and the property that the function $m(\zeta)^2|\zeta|^2$ is increasing, we have
\begin{equation*}
\begin{split}
& \left|m(\zeta_1)^2|\zeta_1|^2-m(\zeta_2)^2|\zeta_2|^2\right|\\
\lesssim &  \left| m(\zeta_1)^2-m(\zeta_2)^2\right| |\zeta_1|^2 +m(\zeta_2)^2|\zeta_{12}||\zeta_{14}||\cos\angle(\zeta_{12},\zeta_{14})|+m(\zeta_1)^2|\zeta_3||\zeta_{12}|\\
 \lesssim & m(\zeta_1)^2|\zeta_1|\left||\zeta_1|-|\zeta_2|\right|+m(\zeta_1)^2|\zeta_{12}||\zeta_1||\cos\angle(\zeta_{12},\zeta_{14})|+m(\zeta_1)^2|\zeta_3||\zeta_{12}|\\
 \lesssim & m(\zeta_1)^2|\zeta_1||\zeta_{12}||\cos\angle(\zeta_{12},\zeta_{14})|+m(\zeta_1)^2|\zeta_3||\zeta_{12}|,
\end{split}
\end{equation*}
which is the desired inequality.
\end{proof}

We will now prove the following theorem, which guarantees the growth estimate $\|Iu(t)\|_{H^1(\mathbb{R}\times\mathbb{T})}$.

\begin{theorem}\label{thm:energy-i}
Let $u(t)$ be a solution to \eqref{eq:NLS}.
Then we have
$$
E_I[u](\delta)\le E_I[u](0)+cN^{-2+}
$$
where the constant $c>0$ depends only on $\|Iu\|_{X^{1,1/2+}([0,\delta])}$.
\end{theorem}
 
\begin{proof}
We will use by \eqref{eq:NLS}
\begin{equation}\label{eq:fNLS}
\widehat{u}_t(t,\zeta)=-i|\zeta|^2\widehat{u}(t,\zeta)-\frac{i}{(2\pi)^2}[\widehat{u}*\widehat{\overline{u}}*\widehat{u}](t,\zeta).
\end{equation}
Applying \eqref{eq:fNLS} to \eqref{eq:modified-e}, we have
\begin{equation*}
\begin{split}
& \frac{d}{dt}E_I[u](t) \\
= & -\frac{i}{2}\int|\zeta_1||\zeta_2|(|\zeta_1|^2-|\zeta_2|^2)m(\zeta_1)m(\zeta_2)\widehat{u}(t,\zeta_1)\widehat{\overline{u}}(t,\zeta_2)\\
& +\frac{i}{2(2\pi)^2}\int\left(|\zeta_1|^2m(\zeta_1)^2-|\zeta_2|^2m(\zeta_2)^2\right)\widehat{u}(t,\zeta_1)\widehat{\overline{u}}(t,\zeta_2)\widehat{u}(t,\zeta_3)\widehat{\overline{u}}(t,\zeta_4)\\
& -\frac{i}{4(2\pi)^2}\int \Lambda_4(\zeta_1,\zeta_2,\zeta_3,\zeta_4) (|\zeta_1|^2-|\zeta_2|^2+|\zeta_3|^2-|\zeta_4|^2)\widehat{u}(t,\zeta_1)\widehat{\overline{u}}(t,\zeta_2)\widehat{u}(t,\zeta_3)\widehat{\overline{u}}(t,\zeta_4)\\
& -\frac{i}{4(2\pi)^4}\int\Lambda_6(\zeta_1,\zeta_2,\zeta_3,\zeta_4,\zeta_5,\zeta_6)\widehat{u}(t,\zeta_1)\widehat{\overline{u}}(t,\zeta_2)\widehat{u}(t,\zeta_3)\widehat{\overline{u}}(t,\zeta_4)\widehat{u}(t,\zeta_5)\widehat{\overline{u}}(t,\zeta_6),
\end{split}
\end{equation*}
where
\begin{equation*}
\begin{split}
\Lambda_6(\zeta_1,\zeta_2,\zeta_3,\zeta_4,\zeta_5,\zeta_6)=& \Lambda_4(\zeta_1+\zeta_2+\zeta_3,\zeta_4,\zeta_5,\zeta_6)-\Lambda_4(\zeta_1,\zeta_2+\zeta_3+\zeta_4,\zeta_5,\zeta_6)\\
& +\Lambda_4(\zeta_1,\zeta_2,\zeta_3+\zeta_4+\zeta_5,\zeta_6)-\Lambda_4(\zeta_1,\zeta_2,\zeta_3,\zeta_4+\zeta_5+\zeta_6).
\end{split}
\end{equation*}
Upon defining
\begin{equation*}
\begin{split}
\widetilde{\Lambda}_4(\zeta_1,\zeta_2,\zeta_3,\zeta_4)=& \left(m(\zeta_1)^2|\zeta_1|^2-m(\zeta_2)^2|\zeta_2|^2+m(\zeta_3)^2|\zeta_3|^2-m(\zeta_4)^2|\zeta_4|^2\right)\\
& 1_{|\cos\angle(\zeta_{12},\zeta_{14})|\le \theta_0~\mathrm{and}~\max\{|\zeta_j|\mid 1\le j\le 4\}\gtrsim N}(\zeta_1,\zeta_2,\zeta_3,\zeta_4),
\end{split}
\end{equation*}
we obtain
\begin{equation}\label{eq:diff}
\begin{split}
& \frac{d}{dt}E_I[u](t) \\
= & -\frac{i}{4(2\pi)^2}\int \widetilde{\Lambda}_4(\zeta_1,\zeta_2,\zeta_3,\zeta_4)\widehat{u}(t,\zeta_1)\widehat{\overline{u}}(t,\zeta_2)\widehat{u}(t,\zeta_3)\widehat{\overline{u}}(t,\zeta_4)\\
& -\frac{i}{4(2\pi)^4}\int\Lambda_6(\zeta_1,\zeta_2,\zeta_3,\zeta_4,\zeta_5,\zeta_6)\widehat{u}(t,\zeta_1)\widehat{\overline{u}}(t,\zeta_2)\widehat{u}(t,\zeta_3)\widehat{\overline{u}}(t,\zeta_4)\widehat{u}(t,\zeta_5)\widehat{\overline{u}}(t,\zeta_6).
\end{split}
\end{equation}

For the sake of the argument, we can assume all functions $\widehat{u}(t,\zeta_j)~(j=1,3,5)$, $\widehat{\overline{u}}(t,\zeta_k)~(k=2,4,6)$ in \eqref{eq:diff} are positive.
Therefore, the proof of the claim reduces to show the following estimates
\begin{equation}\label{eq:e-4}
\left|\int_{\mathbb{R}}\int_* \widetilde{\Lambda}_4(\zeta_1,\zeta_2,\zeta_3,\zeta_4)\widehat{u}(t,\zeta_1)\widehat{\overline{u}}(t,\zeta_2)\widehat{u}(t,\zeta_3)\widehat{\overline{u}}(t,\zeta_4)\,dt\right|\lesssim \frac{1}{N^{2-}}\|Iu\|_{X^{1,1/2+}}^4
\end{equation} 
and
\begin{equation}\label{eq:e-6}
\begin{split}
& \left|\int_{\mathbb{R}}\int_*\Lambda_6(\zeta_1,\zeta_2,\zeta_3,\zeta_4,\zeta_5,\zeta_6)\widehat{u}(t,\zeta_1)\widehat{\overline{u}}(t,\zeta_2)\widehat{u}(t,\zeta_3)\widehat{\overline{u}}(t,\zeta_4)\widehat{u}(t,\zeta_5)\widehat{\overline{u}}(t,\zeta_6)\,dt\right|\\
\lesssim  & \frac{1}{N^{2-}}\|Iu\|_{X^{1,1/2+}}^6,
\end{split}
\end{equation}
where we prefer to state general setting, that is, replacing $\int_0^{\delta}$ by $\int_{\mathbb{R}}$.

As previously observed, by a Littlewood-Palay decomposition, we again use a decomposition in the space of frequencies arising from dyadic partitions of unity, that is $|\zeta_j|\sim N_j$.

\underline{Estimate on \eqref{eq:e-4}:}
We first show the estimate \eqref{eq:e-4}.
By symmetry, we may assume $N_1=\max\{N_j\mid 1\le j\le 4\}\gtrsim N$ and $N_2\ge N_4$. 
We split the case into two sub-cases; $N_1\sim N_3\gg  N_2\gtrsim N_4$, $N_1\sim N_2\gtrsim \max\{N_3,N_4\}$.

In the case when $N_1\sim N_3\gg  N_2\gtrsim N_4$, we have $\cos\angle (\zeta_{12},\zeta_{14})\sim 1+O(N_2/N_1)\sim  1$, which contracts with the constrain $|\cos\angle  (\zeta_{12},\zeta_{14})|\le \theta_0\ll 1$ in the support of $\widetilde{\Lambda}_4(\zeta_1,\zeta_2,\zeta_3,\zeta_4)$.

If $N_1\sim N_2\gtrsim \max\{N_3,N_4,N\}$, by symmetry and Lemma \ref{lem:m1}, we may assume $N_1\sim N_2\gtrsim N_3\ge N_4$ and $N_1 \theta_0\sim 1$, which shows
$$
|\widetilde{\Lambda}_4(\zeta_1,\zeta_2,\zeta_3,\zeta_4)|\lesssim m(N_1)^2N_3^2.
$$
Also if $N_4\gtrsim N$, we use the $L^4$-Strichartz estimate in \eqref{eq:4-} to have that the contribution of this case to the left-hand side of \eqref{eq:e-4} has the bound 
\begin{equation*}
\begin{split}
& c\sum_{N_1\sim N_2\ge N_3\ge N_4\gtrsim N} \frac{m(N_1)^2N_3^2}{N_1N_2N_3N_4}\|u_{N_1}\|_{X^{1,1/2+}}\|u_{N_2}\|_{X^{1,1/2+}}\|u_{N_3}\|_{X^{1,1/2+}}\|u_{N_4}\|_{X^{1,1/2+}}\\
\lesssim & \frac{1}{N^{2-}}\|Iu||_{X^{1,1/2+}}^4,
\end{split}
\end{equation*}
which is acceptable.
Meanwhile, in the case when $N_3\sim N_4$, we have that the contribution of this case to the left-hand side of \eqref{eq:e-4} has the bound 
\begin{equation*}
\begin{split}
& c\sum_{\scriptstyle N_1\sim N_2\gtrsim  N_3\sim N_4 \atop{\scriptstyle N_1\gtrsim N}}\frac{m(N_1)^2N_3^2}{N_1N_2N_3N_4}
\|u_{N_1}\|_{X^{1,1/2+}}\|u_{N_2}\|_{X^{1,1/2+}}\|u_{N_3}\|_{X^{1,1/2+}}\|u_{N_4}\|_{X^{1,1/2+}}\\
\lesssim &\frac{1}{N^{2-}}\|Iu\|_{X^{1,1/2+}}^4.
\end{split}
\end{equation*}
Then we only consider the case when $N_4\ll \min\{N_3,N\}$ in the following.

In the case when $N_3\gg N_4$, the conditions $N_1\sim N_2\gtrsim N$ and $N_3\gg N_4$ allow us to have
$$
|\cos\angle(\zeta_1,\zeta_3)|\le |\cos\angle (\zeta_{12},\zeta_{34})|+O\left(\frac{N_4}{N_3}\right)\le   \theta_0+O\left(\frac{N_4}{N_3}\right)=O\left(\frac{N_4}{N_3}\right).
$$
Using the notations $\zeta_j=(\xi_j,\eta_j)\in\mathbb{R}\times\mathbb{Z}$, we see
$$
o(1)=O\left(\frac{N_4}{N_3}\right)=\cos\angle (\zeta_1,\zeta_3)\sim \frac{\xi_1\xi_3+\eta_1\eta_3}{N_1N_3}
$$
so that, at least, either $|\xi_1|\gtrsim N_1$ or $|\xi_3|\gtrsim  N_3$ holds.
Also if $|\xi_1+\xi_2|\gtrsim N_3$, then $|\xi_3|\gtrsim N_3$.
If that is not the case, $|\xi_1+\xi_2|\ll N_3$ implies $|\xi_3|\ll N_3$, so that $|\xi_1|\sim |\xi_2|\sim N_1$.
We split this case into two sub-cases; $|\xi_1+\xi_2|\gtrsim N_3$, $|\xi_1+\xi_2|\ll N_3$.


Consider the case when $|\xi_1+\xi_2|\gtrsim N_3$.
Let $\theta=N_4/N_1$.
We use the angular decomposition $\arg(\zeta_1)=\ell_j\theta+O(\theta)$ to functions $u_{N_1}$ as
$$
u_{N_1}=\sum_{\ell_1}u_{\ell_1,N_1}.
$$
Also, we use the angular decomposition $\arg(\zeta_{34})=\ell_{34}\theta+O(\theta)$ to functions $u_{N_3}\overline{u_{N_4}}$ as
$$
u_{N_3}\overline{u_{N_4}}=\sum_{\ell_{34}}F_{\ell_{34}},
$$
where
\begin{equation*}
\begin{split}
F_{\ell_{34}}(t,z)= \int_{\mathbb{R}^2}\!\int  & e^{-it(\tau_3+\tau_4)-iz\cdot(\zeta_3+\zeta_4)}1_{|\xi_3|\sim N_3}1_{\arg(\zeta_{34})=\ell_{34}\theta+O(\theta)}\\
 & \widehat{u}_{N_3}(\tau_3,\zeta_4)\widehat{\overline{u}}_{N_4}(\tau_4,\zeta_4)\,(d\zeta_3)(d\zeta_4)d\tau_3d\tau_4.
\end{split}
\end{equation*}
Since
$$
\left|\cos\angle(\zeta_{12},\zeta_{1})-\cos\angle(\zeta_{12},\zeta_{14})\right|\lesssim \frac{N_4}{N_1}\sim\theta,
$$
we have
$$
|\cos\angle(\zeta_{34},\zeta_{1})|\lesssim \theta_0+\theta\lesssim \theta,
$$
Then $|\ell_1-\ell_{34}|=\pi/2\theta+O(1)$ or $|\ell_1-\ell_{34}|=3\pi/2\theta+O(1)$ holds, as we pointed out in the proof of Lemma \ref{lem:bilinear} previously.
By applying Proposition \ref{prop:bilinear0} to $u_{\ell_1,N_1}\overline{u_{N_2}}$, we have that the contribution of this case to the left-hand side of \eqref{eq:e-4} has the bound 
\begin{equation}\label{eq:ar}
\begin{split}
c\sum_{N_1\sim N_2\gtrsim N_3+N\gg N_4}\sum_{\scriptstyle  |\ell_1-\ell_{34}|=\pi/2\theta+O(1) \atop{\scriptstyle \mbox{or}~|\ell_1-\ell_{34}|=3\pi/2\theta+O(1)}}& \frac{m(N_1)^2N_3^2}{N_1N_2}\frac{\langle N_1\theta\rangle^{1/2}}{N_3^{1/2}}\\
& \|u_{\ell_1,N_1}\|_{X^{1,1/2+}} \|u_{N_2}\|_{X^{1,1/2+}}\|F_{\ell_{34}}\|_{L^2(\mathbb{R}\times\mathbb{R}\times\mathbb{T})}.
\end{split}
\end{equation}
Summing over $\ell_1,~\ell_{34}$ and using Remark \ref{rem:pro} for $F_{\ell_{34}}$, we arrive at the bound of the left-hand side of \eqref{eq:ar} by
\begin{equation*}
\begin{split}
&  c\sum_{N_1\sim N_2\gtrsim N_3+N\gg N_4}\frac{N_3^2}{N_1N_2N_3N_4}\frac{N_4^{1/2}}{N_3^{1/2}}\frac{N_4^{1/2}}{N_3^{1/2}}\\
& \qquad \|Iu_{N_1}\|_{X^{1,1/2+}} \|Iu_{N_2}\|_{X^{1,1/2+}}\|u_{N_3}\|_{X^{1,1/2+}}\|u_{N_4}\|_{X^{1,1/2+}}\\
\lesssim & \frac{1}{N^{2-}}\|Iu\|_{X^{1,1/2+}}^4.
\end{split}
\end{equation*}

Next consider the case when $|\xi_1+\xi_2|\ll N_3$ with nice properties $|\xi_1|\sim |\xi_2|\sim N_1$ and $|\cos\angle(\zeta_1,\zeta_3)|\lesssim N_4/N_3$.
Using Propositions \ref{prop:bilinearX} and \ref{prop:1} for $u_{N_1}u_{N_3}$ and $u_{N_2}u_{N_4}$, respectively, we have that the contribution of this case to the left-hand side of \eqref{eq:e-4} has the bound 
\begin{equation*}
\begin{split}
& c\sum_{N_1\sim N_2\gtrsim N_3+N\gg N_4}\frac{m(N_1)^2N_3^2}{N_1N_2N_3N_4}\frac{\langle N_1N_4/N_3\rangle^{1/2}}{N_1^{1/2}}\frac{N_4^{1/2}}{N_1^{1/2}}\\
& \|u_{N_1}\|_{X^{1,1/2+}} \|u_{N_2}\|_{X^{1,1/2+}}\|u_{N_3}\|_{X^{1,1/2+}}\|u_{N_4}\|_{X^{1,1/2+}}\\
\lesssim & \frac{1}{N^{2-}}\|Iu\|_{X^{1,1/2+}}^4,
\end{split}
\end{equation*}
%
thus proving the estimate \eqref{eq:e-4}.

\underline{Estimate on \eqref{eq:e-6}:}
Second, we show the estimate \eqref{eq:e-6}.
Observe that $\max\{|\zeta_j|\mid 1\le j\le 6\}\ll  N$, then $\Lambda_6(\zeta_1,\zeta_2,\zeta_3,\zeta_4,\zeta_5,\zeta_6)=0$.
Then we may suppose $\max\{|\zeta_j|\mid 1\le j\le 6\}\gtrsim N$.
Since $\zeta_1+\zeta_2+\zeta_3+\zeta_4+\zeta_5+\zeta_6=0$ and $\max\{|\zeta_j|\mid 1\le j\le 6\}\gtrsim N$, we have that at least two of $|\zeta_j|$ are greater than $cN$.
With Lemma \ref{lem:m}, we get
\begin{equation}\label{eq:estimate6}
|\Lambda_6(\zeta_1,\zeta_2,\zeta_3,\zeta_4,\zeta_5,\zeta_6)|\lesssim \frac{\max\{m(\zeta_j)^2\mid 1\le j\le 6\}}{\theta_0}.
\end{equation}

If at least three of  $|\zeta_j|$ are larger than $cN$, we assume $|\zeta_1|\gtrsim  |\zeta_2|\gtrsim |\zeta_3|\gtrsim N$ for simplicity.
We take the $L^4$-Strichartz estimate \eqref{eq:4-} to four functions $u_{N_j}$ for $1\le j\le 4$. 
Then by Sobolev's inequality $X^{1+,1/2+}\hookrightarrow L^{\infty}(\mathbb{R}\times \mathbb{R}\times \mathbb{T})$ to two functions $u_{N_j},~j=5,6$, we immediately have that the contribution of this case to the left-hand side of \eqref{eq:e-6} has the bound as
\begin{equation*}
\begin{split}
& c\sum_{N_1\gtrsim N_2\gtrsim N_3\gtrsim N+N_4+N_5+N_6}\frac{1}{\theta_0 N_1N_2N_3N_4}\|Iu_{N_1}\|_{X^{1,1/2+}}\|Iu_{N_2}\|_{X^{1,1/2+}}\prod_{j=3}^6\|u_{N_j}\|_{X^{1,1/2+}}\\
\lesssim & \frac{1}{N^{2-}}\|Iu\|_{X^{1,1/2+}}^6.
\end{split}
\end{equation*}

Let consider the case when two of $|\zeta_j|$ are larger than $cN$ and the other four $|\zeta_j|$ are much smaller than $N$.
Clearly,
\begin{equation}\label{eq:c1}
|\Lambda_4(\zeta_1,\zeta_2,\zeta_3,\zeta_4)|\le 1
\end{equation}
for $\max\{|\zeta_j|\mid 1\le j\le 4\}<N$.
By raising the estimate in Lemma \ref{lem:m},
\begin{equation}\label{eq;c2}
|\Lambda_4(\zeta_1,\zeta_2,\zeta_3,\zeta_4)|\lesssim \max\{m(\zeta_j)^2\mid 1\le j\le 4\},
\end{equation}
if $\min\{|\zeta_1|,|\zeta_3|\}\gtrsim N\gg \max\{|\zeta_2|,|\zeta_4|\}$ or $\max\{|\zeta_1|,|\zeta_3|\}\ll N\lesssim \min\{|\zeta_2|,|\zeta_4|\}$, since from
$$
\left| |\zeta_1|^2-|\zeta_2|^2+|\zeta_3|^2-|\zeta_4|^2  \right|\gtrsim \max\{|\zeta_j|^2\mid 1\le j\le 4\}.
$$
Also by Lemma \ref{lem:m1}, we see 
\begin{equation}\label{eq:c3}
\begin{split}
& |\Lambda_4(\zeta_1,\zeta_2,\zeta_3,\zeta_4)|\\
\lesssim &  \max\{m(\zeta_j)^2\mid 1\le j\le 4\}\left(1+\frac{\min\{|\zeta_1|,|\zeta_3|\}+\min\{|\zeta_2|,|\zeta_4|\}}{\left(\max\{|\zeta_1|,|\zeta_3|\}+\max\{|\zeta_2|,|\zeta_4|\}\right)\theta_0}\right)\\
\sim &  \max\{m(\zeta_j)^2\mid 1\le j\le 4\}\left(1+\min\{|\zeta_1|,|\zeta_3|\}+\min\{|\zeta_2|,|\zeta_4|\}\right)
\end{split}
\end{equation}
for $\max\{|\zeta_1|,|\zeta_3|\}\sim \max\{|\zeta_2|,|\zeta_4|\}\gtrsim N\gg \min\{|\zeta_1|,|\zeta_3|\}+\min\{|\zeta_2|,|\zeta_4|\}$.
Recall the previous consideration.
Suppose $|\zeta_1|\gtrsim  |\zeta_2|\gtrsim N\gg |\zeta_3|\gtrsim |\zeta_4|\gtrsim |\zeta_5|\gtrsim |\zeta_6|$ for simplicity.
Undoing the $L^4$-Strichartz estimate \eqref{eq:4-} to four functions $u_{N_j}$ for $1\le j\le 4$, Sobolev's inequality to two functions $u_{N_j},~j=5,6$, we immediately have that the contribution of this case to the left-hand side of \eqref{eq:e-6} has the bound as
\begin{equation*}
\begin{split}
& c\sum_{N_1\gtrsim N_2\gtrsim N\gg N_3\gtrsim N_4\gtrsim N_5\gtrsim N_6}\frac{N_3}{N_1N_2N_3N_4}\|Iu_{N_1}\|_{X^{1,1/2+}}\|Iu_{N_2}\|_{X^{1,1/2+}}\prod_{j=3}^6\|u_{N_j}\|_{X^{1,1/2+}}\\
\lesssim & \frac{1}{N^{2-}}\|Iu\|_{X^{1,1/2+}}^6.
\end{split}
\end{equation*}
In all other cases, the conclusion easily follows, thus proving our second estimate \eqref{eq:e-6}.
\end{proof}

\begin{remark}
If $m(\zeta)\equiv 1$ and $\theta_0(\zeta_1,\zeta_2,\zeta_3,\zeta_4) \equiv 0$, then $\widetilde{\Lambda}_4=\Lambda_6=0$ in formally, which implies $dE_I[u](t)/dt=0$.
Then $E_I$ is amended the energy conservation for high-frequency levels.
\end{remark}

Let us now in a position to prove Theorem \ref{thm:growth}.
We will obtain an a priori estimate using Lemma \ref{lem:error}, Theorems \ref{thm:LWP} and \ref{thm:energy-i}.

\begin{proof}[Proof of Theorem  \ref{thm:growth}]

Let $u(t)$ be global solution to \eqref{eq:NLS}.
Choose $T>0$ arbitrary.
Split the time interval $[0,T]$ by $[0,\delta_0]$, where $\delta_0=O(1)$ is obtained in Theorem \ref{thm:LWP}.
By Lemma \ref{lem:error}, Theorems \ref{thm:LWP} and \ref{thm:energy-i}, we have
$$
E_I[u](T)\lesssim 1
$$
provided if $T/N^{2-}\lesssim 1$. 
By \eqref{eq:HIH}
$$
\|u(T)\|_{H^s(\mathbb{R}\times\mathbb{T})}\lesssim N^{s-1}\|Iu(T)\|_{H^1(\mathbb{R}\times\mathbb{T})}\lesssim N^{s-1},
$$
hence proving
$$
\|u(T)\|_{H^s(\mathbb{R}\times\mathbb{T})}\lesssim \langle T\rangle^{\frac{s-1}{2}+}.
$$
The proof is complete.
\end{proof}

\end{document}